\documentclass[11pt, a4paper, DIV=14]{scrartcl} 

\usepackage{lmodern}

\usepackage[english]{babel}

\usepackage{amssymb,amsfonts}

\usepackage{comment}

\usepackage{amsmath, amsthm, color, graphicx} 
\usepackage{subfigure}

\usepackage{mathtools}

\definecolor{dullmagenta}{rgb}{0.4,0,0.4}   
\definecolor{darkblue}{rgb}{0,0,0.4}
\definecolor{darkgreen}{rgb}{0,0.6,0}
\definecolor{darkred}{rgb}{0.6,0,0}

\newtheorem {theorem}{Theorem}

\newtheorem {proposition}[theorem]{Proposition}
\newtheorem {corollary}[theorem]{Corollary}

\newtheorem {defi}{Definition}

\newtheorem {remark}{Remark}

\numberwithin{equation}{section}
\numberwithin{theorem}{section}
\numberwithin{defi}{section}
\numberwithin{remark}{section}
\numberwithin{figure}{section}

\newcommand{\bN}{{\mathbb N}} 
\newcommand{\bZ}{{\mathbb Z}} 
\newcommand{\bQ}{{\mathbb Q}} 
\newcommand{\bR}{{\mathbb R}} 

\newcommand{\bS}{{\mathbb S}} 

\newcommand{\cZ}{{\cal Z}}

\newcommand{\bem}{\left(\!\begin{array}}
\newcommand{\eem}{\end{array}\!\right)}
\newcommand{\bsm}{\left(\begin{smallmatrix}} 
\newcommand{\esm}{\end{smallmatrix}\right)}  

\newcommand{\pa}{\partial}

\DeclareMathOperator{\xCinfc}{{C^\infty_c}}
\DeclareMathOperator{\xLtwo}{{L^2}}
\DeclareMathOperator{\dom}{D}

\DeclareMathOperator{\vol}{{d\omega}}
\DeclareMathOperator{\lapl}{{\Delta}}

\title{Spectral analysis and the Aharonov-Bohm~effect on certain almost-Riemannian manifolds}
\author{U. Boscain\thanks{Ugo Boscain,  CNRS, CMAP, \'Ecole Polytechnique, Palaiseau,   France,
\& Team GECO, INRIA Saclay,
{\tt ugo.boscain@polytechnique.edu}},
D. Prandi\thanks{ D.Prandi, CEREMADE, Universit\'e Paris-Dauphine, France,  
{\tt dario.prandi@univ-tln.fr}}  
and 
M. Seri\thanks{Marcello Seri, 
Department of Mathematics
University College London
Gower Street, London WC1E 6BT, United Kingdom,
{\tt m.seri@ucl.ac.uk}}
}
\date{\today}

\begin{document}

\maketitle
\begin{abstract}
  We study spectral properties of the Laplace-Beltrami operator on two relevant almost-Riemannian manifolds, namely the Grushin structures on the cylinder and on the sphere. 
  This operator contains first order diverging terms caused by the divergence of the volume.

  We get explicit descriptions of the spectrum and the eigenfunctions. In particular in both cases we get a Weyl's law with leading term $E\log E$.
  We then study the drastic effect of Aharonov-Bohm magnetic potentials on the spectral properties. 

  Other generalised Riemannian structures including conic and anti-conic type manifolds are also studied.
  In this case, the Aharonov-Bohm magnetic potential may affect the self-adjointness of the Laplace-Beltrami operator.
\end{abstract}


\section{Introduction}

A $2$-dimensional almost-Riemannian structure is a generalized Riemannian structure that can be locally defined by a pair of smooth vector fields on a $2$-dimensional manifold, satisfying the H\"ormander condition  (see for instance \cite{AgrBarBoscbook,bellaiche}).   
These vector fields play the role of an orthonormal frame. 
Where the linear span of the two vector fields is $2$-dimensional, the corresponding metric is Riemannian. 
Where it is $1$-dimensional, the corresponding Riemannian metric is not well-defined.
Generically this happens on a $1$-dimensional submanifold denoted by $\cZ$.
However, thanks to the H\"ormander condition, one can still define a distance between two points (called Carnot-Carath\'eodory distance), which happens to be finite and continuous. 
See \cite{gaussbonnet1,arsnormal}.

Almost-Riemannian structures were introduced in the context of hypoelliptic operators \cite{baouendi,FL1,grushin1}. 
They appeared in  problems of  population transfer in quantum systems  \cite{q1,BCha} and have applications to  orbital transfer in space mechanics  \cite{tannaka}. From a theoretical point of view, they present interesting phenomena. 
For instance, the singular set acts as a barrier for the heat flow and for a quantum particle, even though geodesics can pass through the singular set without singularities \cite{BL11,BP13}. 

Not much is known about spectral properties of the Laplace-Beltrami operator on almost-Riemannian structures.
The only exception is the discreteness of the spectrum under some general assumptions in the compact case, proved in \cite{BL11}.
We remark that this result is not trivial since the considered structures, when genuinely almost-Riemannian, have always infinite volume.
Hypoelliptic Laplacians associated with rank-varying sub-Riemannian structures, of which almost-Riemannian structures are a particular case, have been studied by Montgomery \cite{montrev}.
In particular he presented an interesting analysis of the ground state of these operators and their connection with the presence of abnormal extremals.
He also provided a remarkable interpretation of the problem in terms of a quantum particle in a magnetic field.

In this paper we study the spectral properties of the Laplace-Beltrami operator $\Delta$ in two relevant almost-Riemannian structures with infinite volume: the Grushin cylinder and the Grushin sphere.
The former is a compactification along the $y$ axis of the well-known Grushin plane, while the latter has been first defined in \cite{BCha}.
In particular, we obtain an explicit description of the spectrum, the eigenfunctions, and the corresponding Weyl's law for the counting function of the discrete part of the spectrum:
\begin{equation}\label{eq:weyl-counting}
N(E) := \#\{ \lambda\in\sigma_{\rm d}(-\Delta) \mid \lambda\le E\}.
\end{equation}
Remarkably the leading order turns out to be $E\log(E)$, which is fairly unusual for Laplace-Beltrami operators on $2$-dimensional Riemannian manifolds \cite{B07,  I80, P95, S02}.

In these settings, we also investigate the effects of hidden magnetic fluxes on the spectrum.
To this purpose, we introduce a magnetic vector potential for the zero magnetic field whose flux is non zero.
In the Euclidean case, it is well understood that these fluxes, despite being classically invisible, affect the wave functions by introducing a change of phase. 
This is known as Aharonov-Bohm effect \cite{AT98, dOP08} and requires the space to be non simply connected (one can imagine that these hidden fluxes are generated by fields defined only in the points removed from the space).

The picture changes completely for asymptotically hyperbolic manifolds with cusps. 
In such cases, as proved in \cite{GM07}, a change in vector potentials (that preserves the magnetic fields, including the zero field) can drastically modify the spectral properties of the operator, e.g.\ by destroying the absolutely continuous component of the spectrum. In this work we show that the same phenomenon is present in almost-Riemannian manifolds.
While it is known that such phenomena can appear in presence of magnetic degenerations \cite{hmmrev, jrev}, it is very surprising to see them triggered simply by the variation of a magnetic parameter.

Due to the explicit nature of our examples, we are able to define a continuous parametrisation of the Ahronov-Bohm vector potentials. 
This enable us to follow closely the mechanisms of spectral accumulation in the limit in which a trapping vector potential (i.e.\ a vector potential for which the spectrum is purely discrete) becomes non-trapping (i.e.\ a vector potential for which the absolutely continuous spectrum is present). 
In particular we can explicitly describe the family of eigenfunctions that degenerates into generalised eigenfunctions.

Finally, with the same techniques we are also able to study spectral properties of the Laplace-Beltrami operator on conic and anti-conic surfaces of the type considered in \cite{BP13}.
Here, the Aharonov-Bohm effect affects not only the spectrum but also the self-adjointness properties of the operator.

Explicit results similar to the ones presented in this paper can be obtained also for Laplace-Beltrami operators on certain examples of $n$-dimensional almost-Riemannian manifolds. We omit the details of such cases but the few known explicit examples seem to point to a general behaviour for the sub-Riemannian Weyl's law.

\vspace{1em}
\textbf{Open problem.}
Given an $n$-dimensional almost-Riemannian manifold $M$ for which the associated Laplace-Beltrami operator $\Delta$ is essentially self-adjoint on a connected component $\Omega$ of $M\setminus \cZ$, it would be interesting to relate the Weyl law for the discrete part of the spectrum of $\Delta$ to the Hausdorff dimensions of $\Omega$ and of $\partial  \Omega$.
This problem is connected to the recent papers \cite{cdv15, ghezzi15}.

Unfortunately, the techniques used in this paper rely heavily on the integrability of the geodesic equation of the structures under consideration, and are not suitable to study generic almost-Riemannian manifolds.
See also the related paper \cite{BL11}.




\vspace{1em}
\textbf{Structure of the paper.}
In Section \ref{sec:settingsandresults} we introduce the setting and our results.
Section \ref{sec:grushin2d} contains the proof of the theorems for the Grushin cylinder. 
Section \ref{sec:quantum} contains the proof of the theorems for the Grushin sphere.

\vspace{1em}
\textbf{Notation.}
We use the conventions $\bR^* = \bR\setminus\{0\}$, $0\not\in\bN$ and $\bN_0 = \{0\}\cup \bN$.
The ceiling and floor functions are respectively denoted by $\lceil\cdot\rceil$ and $\lfloor\cdot\rfloor$. 

\section{Setting and main results}\label{sec:settingsandresults}

\begin{figure}
  \begin{center}
    \subfigure[geodesics starting from the singular set (red circle), up to t=1.7. The black (self-intersecting) curve line is the wave front (i.e., the end point of all geodesics at time 1.7)]{\includegraphics[width=.45\textwidth]{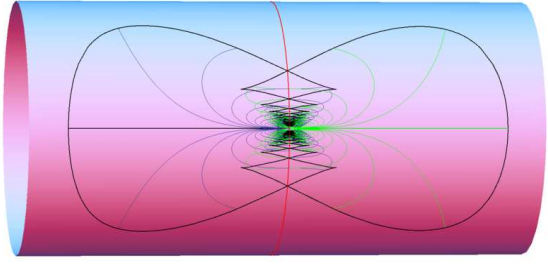}\label{fig1}}
    \hfill
    \subfigure[geodesics starting from the the point (0.3,0). Notice that they cross the singular set (red circle) with no singularities.]{\label{fig2}\includegraphics[width=.48\textwidth]{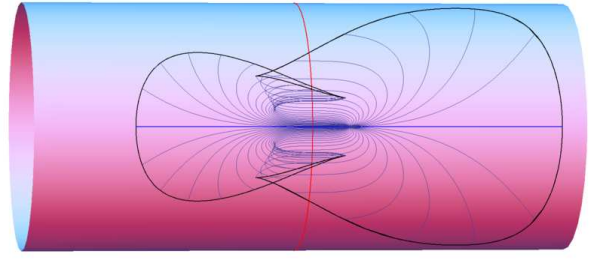}}
    \caption{\small Geodesics for the Grushin cylinder for $t\in[0, 1.7]$. For their analytic expression see \cite{BL11}.}
  \end{center}
\end{figure}

In this paper we consider only trivializable almost-Riemannian structures. 
For a more general definition of (possibly non-trivializable) almost-Riemannian structures we refer to \cite{AgrBarBoscbook, gaussbonnet1, gaussbonnet2}.

\begin{defi}
  \label{d-ars}
  A \emph{trivializable almost-Riemannian structure} is a triple $(M, X_1, X_2)$ where $M$ is a $2$-dimensional smooth manifold and $X_1, X_2$ are two smooth vector fields satisfying the H\"ormander condition.
\end{defi}

From now on, we drop the term ``trivializable''.
The pair $(X_1,X_2)$ is called the \emph{generating frame} of the almost-Riemannian structure.
We denote the span of $X_1$ and $X_2$ by $\blacktriangle(q):=\mbox{span}\{X_1(q),X_2(q)\}$. The set where $\blacktriangle(q)$ is one dimensional is called the {\em singular set} and it is denoted by ${\cal Z}$.
Notice that, thanks to the H\"ormander condition, $\blacktriangle(q)$ is either one or two dimensional. 
Generically the singular set is  a one dimensional sub-manifold of $M$ and beside isolated points $\blacktriangle(q)$, $q\in\cZ$, is not tangent to $\cZ$ \cite{arsnormal}.

For $q_1, q_2\in M$ it is possible to define a natural concept of distance, called Carnot-Carath\'eodory distance. 
This endows the manifold with a metric structure compatible with its original topology, see \cite{AgrBarBoscbook, bellaiche, montgomery}. 
On $M\setminus\cZ$, the Carnot-Carath\'eodory distance is the Riemannian distance for which $(X_1, X_2)$ is an orthonormal frame.

The presence of a Riemannian metric on  $M\setminus\cZ$ allows for a natural concept of volume -- the Riemannian volume $d\omega$ -- which diverges while approaching $\cZ$.
We can then define the Laplace-Beltrami operator $\Delta$ on $L^2(M\setminus \cZ,d\omega)$ in the standard way
\begin{equation}\label{eq:lapl-form}
\Delta := d^* d, \qquad
d^* := \star d \star,
\end{equation}
where $d$ is the exterior differential and $\star$ is the Hodge star operator. 
Analogously we define the magnetic Laplace-Beltrami operator associated with a magnetic vector potential $A\in \Omega^1(M)$ as $\Delta_A := (d + A)^*(d+A)$ on $L^2(M\setminus \cZ, d\omega)$. 
The associated magnetic field is $B = dA$.

Let $(M, X_1, X_2)$ be an almost-Riemannian structure and $(Y_1, Y_2)$ be a rotation of $(X_1,X_2)$, possibly depending smoothly on the point.
Then, we say that $(M, X_1, X_2)$ and $(M,Y_1,Y_2)$ are two equivalent almost-Riemannian structures. 
In particular, they define the same $\blacktriangle$ and the same Carnot-Carth\'eodory distance.
For every point $q_0\in M$ there always exists a local system of coordinates $(x,y)$ on a neighborhood $U$ of $q_0$ such that the almost-Riemannian structure $(M|_U,X_1,X_2)$ is equivalent to $(V, Y_1,Y_2)$, $V\subset \bR^2$, where
\[
Y_1(x,y) =\left(
\begin{array}[c]{c}
  1 \\ 0
\end{array}\right),
\qquad
Y_2(x,y) =\left(
\begin{array}[c]{c}
  0 \\ f(x,y)
\end{array}\right),
\]
for some smooth function $f$. (See \cite{gaussbonnet1}.)
With this choice, the singular set $\cZ$ is the zero-level set of $f$.
On $U\setminus\cZ$ the Riemannian metric $g$, the area element $d\omega$, and the Laplace-Beltrami operator $\Delta$ are
\begin{gather*}
  g(x,y) =
  dx^2 + \frac{dy^2}{f(x,y)^2}, \quad
  d\omega = \frac {dx\,dy} {|f(x,y)|}, \quad
  \Delta  u = \pa_x^2 u + f^2\pa_y^2 u - \frac{\pa_x f}{f}\pa_x  u + f(\pa_y f)\pa_y u.
\end{gather*}

\begin{table}
  \centering
  \begin{tabular}{|c||c|c|}
    \hline
    & Grushin cylinder & Grushin sphere \\  \hline\hline \\[-1em]
  $M$ & $\bR\times\bS^1$ & $\bS^2$ \\ [1em] \hline \\[-1em]
   &  & $x\in (-\pi/2,\pi/2)$, $\phi\in[0,2\pi]/\sim$ \\
	{\small Coordinates} & $x\in \bR$, $\theta\in[0,2\pi]/\sim$ & {\small These coordinates are singular in } $x=\pm\pi/2$ \\ \hline \\[-1em]
  $(X_1,X_2)$ & $X_1(x,\theta) = \left( \begin{array}{c} 1 \\ 0 \end{array} \right)$,   $X_2(x,\theta) = \left( \begin{array}{c} 0 \\ x \end{array} \right)$  & $X_1(x,\phi) = \left( \begin{array}{c} 1 \\ 0 \end{array} \right)$,   $X_2(x,\phi) = \left( \begin{array}{c} 0 \\ \tan x \end{array} \right)$ \\ [1em] \hline \\[-1em]
  {\small Singular} &  &  $\{x=0\}$ \\
  {\small set }$\cZ$ & $\{x = 0\}$ & {\small The singularity $\{x=\pm \pi/2\}$}\\ 
   &  &    {\small is due to the system of coordinates.}   \\ \hline  \\[-1em]
  {\small Volume }$d\omega$ & $\displaystyle\frac 1 {|x|}dx\,d\theta$ &  $\displaystyle\frac 1 {|\tan(x)|}dx\,d\phi$ \\ [1em] \hline  \\[-1em]
 {\small Laplace-Beltrami} $\Delta$ & $\partial_x^2 - \displaystyle\frac 1 x \partial_x +x^2\partial_\theta^2$ & $\partial_x^2 - \displaystyle\frac 1 {\sin(x) \cos(x)} \partial_x + \tan(x)^2\partial_\phi^2$ \\ [1em] \hline 
  \multicolumn{3}{|c|}{\small The Grushin sphere can be defined as $\bS^2 = \{ y_1^2+y_2^2+y_3^2 = 1\}$ with the}\\
  \multicolumn{3}{|c|}{\small orthonormal frame $Y_2=(0, -y_3, y_2)^T$, $Y_1=(-y_3, 0, y_1)^T$. In this case, }\\
  \multicolumn{3}{|c|}{\small $\cZ = \{y_3 = 0, y_1^2+y_2^2 = 1\}$. The above representation is obtained by }\\
  \multicolumn{3}{|c|}{\small passing in spherical coordinates $(\cos x \cos\phi, \cos x\sin\phi, \sin x)$ and letting}\\
\multicolumn{3}{|c|}{\small 
  $X_1 = \cos(\phi-\pi/2) Y_1 - \sin(\phi-\pi/2) Y_2 = (1,0)^T$,}\\
 \multicolumn{3}{|c|} {\small $X_2 = \sin(\phi-\pi/2) Y_1 + \cos(\phi-\pi/2) Y_2 = (0, \tan x)^T$.}  \\ \hline
\end{tabular}
\caption{\small The Grushin cylinder and the Grushin sphere.}
\label{tab:ar}
\end{table}

In this paper we are going to focus mainly on the Grushin cylinder and the Grushin sphere presented in Table~\ref{tab:ar}. 

\subsection{Spectra of the Laplace-Beltrami operators}

Let $M$ be the Grushin cylinder.
In \cite{BL11} it has been shown that the Laplace-Beltrami operator, with domain $\xCinfc(M\setminus\cZ)$, is essentially self-adjoint on $L^2(M, d\omega)$ and separates in the direct sum of its restrictions to $M_\pm = \bR_\pm\times\bS^1$.
Therefore, without loss of generality we focus on $\lapl$ on $M_+$.

\begin{theorem}[Grushin cylinder case]
  \label{thm:spec-grushin2d}
  The operator $-\Delta$ on $L^2(M_+, d\omega)$, defined in Table~\ref{tab:ar}, has absolutely continuous spectrum $\sigma(-\Delta) = [0,\infty)$ with embedded discrete spectrum 
  \begin{equation*}
    \sigma_{\rm d}(-\Delta) = \left\{ \lambda_{n,k} = 4|k|n \mid n\in\bN,\, k\in\bZ\setminus\{0\} \right\}.
  \end{equation*}
  The corresponding eigenfunctions are given by
  $\psi_{n,k}(x,\theta) = e^{ik\theta} \dfrac{1}{x}W_{n,\frac{1}{2}}(|k|x^2)$,
  where $W_{\nu,\mu}$ is the Whittaker $W$-function with parameters $\nu$ and $\mu$.
\end{theorem}

With the above explicit description, it is possible to calculate the Weyl's law for the Laplacian on the Grushin cylinder, intended as the asymptotic number of eigenvalues and embedded eigenvalues below the threshold energy $E$.

\begin{corollary}[Grushin cylinder case]
  \label{cor:counting-grushin2d}
  The Weyl's law with remainder as $E\rightarrow+\infty$ is
  \begin{equation*}
    N(E) = \frac{E}2\log(E) + \left(\gamma-2\log(2)\right) \frac{E}2 + O(1),
  \end{equation*}
  where $\gamma$ is the Euler-Mascheroni constant.
\end{corollary}

Similar results hold for the Laplace-Beltrami operator on the Grushin sphere.
As shown in \cite{BL11}, this operator is essentially self-adjoint in $L^2(\bS^2,d\omega)$ and its spectrum is purely discrete.
As in the cylinder case, it separates in the direct sum of its restrictions to the north and south hemispheres $S_\pm$, cut along the equatorial singularity.
Thus, we without loss of generality we consider $\lapl$ on the north hemisphere $S_+$.

\begin{theorem}[Grushin sphere case]
  \label{thm:spec-grushinS}
  The operator $-\Delta$ on $L^2(S_+,d\omega)$, defined in Table~\ref{tab:ar}, has purely discrete spectrum 
  \begin{equation*}
    \sigma_{\rm d}(-\Delta) := \left\{ \lambda_{n,k} := 4n(n+|k|) \mid n\in\bN,\, k\in\bZ \right\}.
  \end{equation*}
  The corresponding eigenfunctions are 
  \[
  \psi_{n,k}(x,\phi) = e^{ik\left(\phi + \frac{\pi}2\right)} \cos(x)^{k} F\left( 
      -(n+1), \;
      n+k+1; \;
      1 + k; \;
      \cos(x)^2
    \right)
  \]
  where $F(a,b;c;x)$ is the Gauss Hypergeometric function with parameters $a,b,c$.
\end{theorem}

\begin{corollary}[Grushin sphere case]
  \label{cor:counting-grushinS}
  The Weyl's law with remainder as $E\rightarrow+\infty$ is
  \begin{equation*}
    N(E) = \frac{E}4\log(E) + \left(\gamma - \log(2) - \frac12 \right)\frac{E}2 + O(\sqrt E),
  \end{equation*}
  where $\gamma$ is the Euler-Mascheroni constant.
\end{corollary}

\subsection{Spectra of the Aharonov-Bohm magnetic Laplace-Beltrami operator}
\label{sec:intro-spec-ab}

To mimic the Ahronov-Bohm effect for the Laplace-Beltrami operator in the Grushin cylinder we consider the vector potential $A_b := -i b\; d\theta$, $b\in\mathbb{R}$. 
The associated magnetic Laplace-Beltrami operator on the Grushin cylinder reads
\begin{equation}
  \label{eq:gr-ab}
  \Delta^b  = \pa_x^2 -\frac{1}x\pa_x +|x|^{2}(\pa_\theta^2 -2ib\;\pa_\theta - b^2).
\end{equation}

\begin{theorem}
  [Grushin cylinder case]
  \label{thm:spec-grushin2d-ab}
  The operator $-\Delta^b$ on $L^2(M_+,d\omega)$ has non-empty discrete spectral component
  \begin{equation}
    \label{eq:spec-grushin2d-ab}
    \sigma_{\rm d}(-\Delta^b) = \left\{ \lambda_{n,k}^b := 4n|k-b| \mid n\in\bN,\, k\in\bZ\setminus\{b\}  \right\}.
  \end{equation}
  When $b\in\bZ$ the operator has in addition absolutely continuous spectrum $[0,+\infty)$.
  When $b\not\in \bZ$ the spectrum has no absolutely continuous part.
  In any case, the eigenfunctions are 
      $\psi_{n,k}^b(x,\theta) = e^{ik\theta} \frac{1}{x}W_{n,\frac{1}{2}}(|k-b|x^2)$.
\end{theorem}

\begin{corollary}
  \label{cor:counting-grushin2d-ab}
  If $b\in\bZ$, the Weyl's law is the one of Corollary~\ref{cor:counting-grushin2d}.
  If $b\not\in\bZ$, let $\kappa\in\bZ$ be the closest integer to $b$. 
  Then, the Weyl's law with remainder as $E\rightarrow+\infty$ is
  \begin{equation*}
    N(E) = \frac{E}2 \log(E) + \frac{E}{2}\left(\frac1{2|\kappa-b|} + \gamma -2\log(2) -\frac{\psi(1-|\kappa-b|)+\psi(1+|\kappa-b|)}{2}\right) + O(1),
  \end{equation*}
  where $\gamma$ is the Euler-Mascheroni constant and $\psi(x)$ is the digamma function.
  Here, the $O(1)$ is uniformly bounded with respect to $b$.
\end{corollary}

\begin{remark}
Notice that $N(E)$ diverges for $b\to\kappa$ since, in this limit, part of the discrete spectrum accumulates into the absolutely continuous component.
\end{remark}

For this operator, we can explicitly describe the degeneracy of the spectrum, as function of $b$.

\begin{theorem}
  [Degeneracy of the spectrum in the Grushin cylinder case]
  \label{thm:degenerspec-grushin2d-ab}
  Let $d(n)$ denote the number of divisors of $n$. 
  Then, 
  \begin{itemize}
  \item If $b\in\bR\setminus\bQ$, the spectrum is simple.
  \item If $b\in\bQ$, the discrete spectrum is degenerate in the following sense: each eigenvalue $\lambda$ has multiplicity bounded from above by $2 d(\lambda/4)$.
  \item If $b\in\bZ$, the eigenvalues achieve the maximal degeneracy and the multiplicity is exactly 
  \begin{equation}
  \label{eq:dlambda}
  \begin{cases}
  2 d(\lambda/4), & \mbox{if $\lambda/4$ is odd},\\
  2 d(\lambda/4) - 2, & \mbox{if $\lambda/4$ is even}.\\
  \end{cases}
  \end{equation}
  \end{itemize}
\end{theorem}

\begin{remark}
A direct consequence of the previous theorem is that the maximal multiplicity of the eigenvalues has very slow growth. In fact, it is well known \cite{A76} that as $n\to\infty$ we have
$d(n) = o(n^\epsilon)$, for any $\epsilon>0$.
\end{remark}

Finally, we can give deeper information on the spectral accumulation in the limit $b\to k$ and the corresponding degeneration of the eigenfunctions.
As an immediate consequence of Theorem~\ref{thm:spec-grushin2d-ab} we have the following.

\begin{corollary}
  [Spectral accumulation on the Grushin cylinder]
  \label{thm:decompspec-grushin2d-ab}
Fix $k\in\bZ$. 
Then, for every $n\in\bN$, the level spacing satisfies
$|\lambda^b_{n,k} - \lambda^b_{n-1,k}| \to 0$  as $b\to k$.
Moreover, for any fixed interval $I=[x_1,x_2]\subset[0,\infty)$ and any $N\in\bN$, we have
$\#\{n\in\bN \mid \lambda_{n,k}^b \in I \} \geq N$  as $b\to k$.
\end{corollary}

\begin{theorem}
  [Degeneration of the eigenfunctions on the Grushin cylinder]
  \label{thm:degenef-grushin2d-ab}
Fix $k\in\bZ$. Then for any $\lambda\in\bQ$, $\lambda>0$, there exist a sequence of pairs $(b_j, n_j)\in\left(k-\frac{1}2,k+\frac{1}2\right)\times\bN$, with $b_j \to k$ and $n_j \to \infty$, such that
\begin{equation}\label{eq:limitEigenFcts}
\psi^{b_j}_{n_j,k}(x,\theta) \to e^{ik\theta} \frac{\sqrt{\lambda}}2 J_1(\sqrt{\lambda}x)
\end{equation}
uniformly on compact sets, where $J_\nu(z)$ is the Bessel function of the first kind of order $\nu$. 
The limit function on the r.h.s.\ is the generalised eigenfunction of $\Delta^b$ with generalized eigenvalue $\lambda$ (see Remark~\ref{rmk:geneigen}).
\end{theorem}

\begin{figure}[t!]
  \centering
  \includegraphics[width=.7\linewidth]{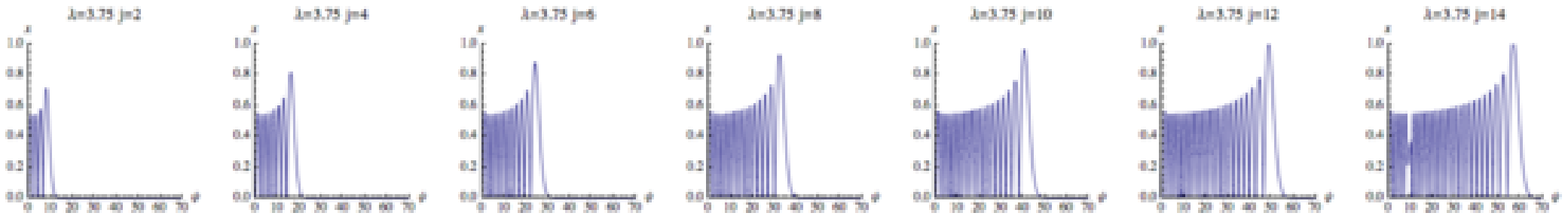}
  \includegraphics[width=.7\linewidth]{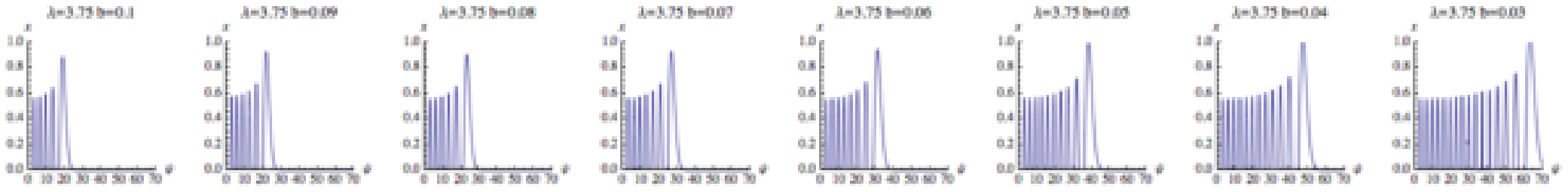}
  \caption{\small The first row shows the spreading of the projection onto $\theta=0$ of $\psi_{n_j,0}^{b_j}(x)$ as $j$ increases for $\lambda=3.75$. The second row shows the spreading of the projection onto $\theta=0$ of $\psi_{n(b),0}^b(x)$ as $b\to 0$ for $\lambda=3.75$. See Theorem~\ref{thm:degenef-grushin2d-ab} and Remark~\ref{rem:degenef-grushin}}
  \label{fig:decompact}
\end{figure}

\begin{remark}\label{rem:degenef-grushin}
  Theorem~\ref{thm:degenef-grushin2d-ab} can be rewritten as follows. 
  For every $\lambda>0$, let $n(b) := 2 \left\lceil \frac{\lambda}{8|b-k|} \right\rceil$.
  Then 
  $
    \lim_{b\to k}\psi_{n(b),k}^b(x,\theta) = e^{i k \theta} \frac{\sqrt{\lambda}}2 J_1(\sqrt{\lambda}x)
  $
  uniformly on compact sets.
  The proof is similar to the one of Theorem~\ref{thm:degenef-grushin2d-ab} with $n_j$ replaced by $n(b)$.
\end{remark}

\begin{figure}[t]
\centering
\includegraphics[width=.7\linewidth]{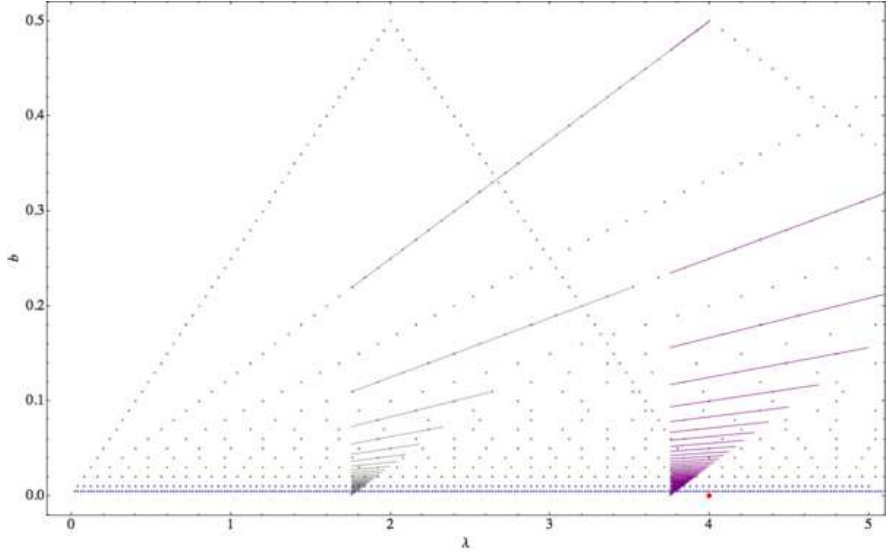}
\caption{\small The dots correspond to the eigenvalues up to energy $5$ for some values of $b$ as it gets closer to $\kappa=0$. The thick red dot represents the only embedded eigenvalue $\lambda=4$ of the operator with $b=0$ up to energy $5$. The grey line is the the curve $\lambda_{n(b),\kappa}^b$ (see Remark~\ref{rem:degenef-grushin}) converging to $1.75$ as $b\to\kappa$. The purple one is the curve $\lambda_{n(b),\kappa}^b$ converging to $3.75$.}
\label{fig:rette}
\end{figure}

For what concerns the Aharonov-Bohm effect on the Grushin sphere, since $\bS^2$ is simply connected any closed form is exact and, hence, we cannot hope to obtain an Aharonov-Bohm effect without artificially poking a hole in the manifold.
This is the same phenomena as in the original Aharonov-Bohm effect \cite{AT98, dOP08}. 

We will thus consider the magnetic Laplace-Beltrami operator induced by the magnetic vector potential $A_b := -ib\; d\phi$ on the north hemisphere of $\bS^2$ with removed north pole $S_+^\circ$ and Dirichlet boundary conditions. Note that on $S_+^\circ$ the corresponding magnetic field is $0$. The corresponding operator is 
\begin{equation}
  \label{eq:ABlaplace-sphere}
  \Delta^b = \pa_x^2 - \frac1{\sin(x)\cos(x)}\pa_x + \tan(x)^2 \left(
  \pa_\phi^2 -2ib \pa_\phi -b^2
  \right).
\end{equation}


\begin{theorem}[Grushin sphere case]
\label{thm:spec-grushinS-ab}
  The operator $-\Delta^b$ defined in \eqref{eq:ABlaplace-sphere} and acting on $L^2(S_+^\circ, d\omega)$, has purely discrete spectrum 
  \[
  \sigma_{\rm d}(-\Delta^b) = \{ \lambda_{n,k}=4n(n+|k-b|) \mid n\in \bN, k\in \bZ \}.
  \]
  The corresponding eigenfunctions are given by
  \[
  \psi_{n,k}(x,\phi)  = e^{ik\phi} e^{i(k-b)\frac \pi 2} \cos(x)^{k-b} \, F\left( -(n+1), n+k-b+1; 1+k-b; cos(x)^2\right).
  \]
\end{theorem}


\begin{corollary}
\label{cor:counting-grushinS-ab}
The Weyl's law with remainder as $E\rightarrow+\infty$ is
  \begin{equation*}
    N(E) = \frac{E}{4}\log(E) + \left( \gamma- \log(2) -\frac 1 2 \right)E + O(\sqrt E),
  \end{equation*}
  where $\gamma$ is the Euler-Mascheroni constant, and the big O is uniformly bounded with respect to\ $b$.
\end{corollary}

Notice that the first two orders of the asymptotic expansion of $N(E)$ are independent of Aharnov-Bohm potential, indeed the parameter $b$ is hidden in the remainder term.

\begin{corollary}[Degeneracy of the spectrum in the Grushin sphere case]
\label{cor:degeneracy-grushinS-ab}
If $b\in\bR\setminus\bQ$ the spectrum is simple, if $b\in\bQ$ the spectrum is finitely degenerate.
\end{corollary}

A brief but more detailed discussion on the degeneracy can be found in Section~\ref{sec:quantum}.

\subsection{Additional results on general conic and anti-conic type surfaces}
\label{sec:conic}

The effect of the Aharonov-Bohm perturbation can be even stronger than the one we saw in almost-Riemannian geometry when considering the more general structures studied in \cite{BP13}.
For any $\alpha\in\bR$ consider on $M:=\bR^*\times\bS^1$ the orthonormal frame
\[
  X_1(x,\theta) := \bem{c} 1 \\ 0 \eem,
  \qquad
  X_2(x,\theta) := \bem{c} 0 \\ |x|^\alpha \eem, 
  \qquad (x,\theta)\in M.
\]
That is, we consider on $M$ the Riemannian metric $g_\alpha = dx^2 + x^{-2\alpha} d\theta^2$.
The associated Riemannian volume is $d\omega_\alpha=|x|^{-\alpha}dxd\theta$.

For any $\alpha\ge 0$, this metric can be completed in $\bR\times\bS^1$ in such a way that the corresponding distance {induces the  topology }of a cylinder. 
In particular, when $\alpha$ is a positive integer this is a trivializable almost-Riemannian structure in the sense of Definition~\ref{d-ars}.
When $\alpha<0$ this metric can be extended in $\bR\times\bS^1/\sim$, where $p\sim q$ if $p=q$ or $p,q\in\{0\}\times\bS^1$.
The corresponding distance induces on  $\bR\times\bS^1/\sim$   {the topology of a cone.}

The Laplace-Beltrami operator on $L^2(M,d\omega_\alpha)$ with the Ahronov-Bohm vector potential $A_b = -i b d\theta$ is
\[
\Delta_b = \pa_x^2 +|x|^{2\alpha}\pa_\theta^2 +|x|^{2\alpha}\left(\pa_\theta^2 -2ib\pa_\theta -b^2  \right),
\]
where in particular $\Delta = \Delta_0$.

As in \cite{BP13}, the Fourier decomposition $L^2(M,d\omega_\alpha) = \bigoplus_{k=0}^\infty H_k$, $H_k \simeq L^2(\bR^*, |x|^{-\alpha}dx)$, yields on each $H_k$ the operator
\begin{equation}
  \label{eq:hatdelta-ab}
  \hat\Delta_{\alpha,k}^b = \pa_x^2-\frac\alpha{x}\pa_x-|x|^{2\alpha} (b-k)^2.
\end{equation}
We recall that $\hat\Delta_{\alpha,k} = \hat\Delta_{\alpha,k}^0$.

\begin{proposition}
  [\cite{BP13}]
  \label{prop:bp13-sa}
  The operator $\Delta_\alpha$ on $L^2(M,d\omega_\alpha)$ with domain $\xCinfc(M)$ is essentially self-adjoint if and only if  $\alpha\ge 1$ or $\alpha\le-3$.
  Moreover, on $\xCinfc(\bR^*)$:
  \begin{itemize}
  \item  if $-3<\alpha\le-1$ for every $k\neq0$ the operator $\hat\Delta_{\alpha,k}$ is essentially self-adjoint, while $\hat\Delta_{\alpha,0}$ is not;
  \item  if $-1<\alpha<1$, every $\hat\Delta_{\alpha,k}$ is not essentially self-adjoint.
  \end{itemize}
\end{proposition}

The proof of Proposition~\ref{prop:bp13-sa} applied to \eqref{eq:hatdelta-ab} yields the following.

\begin{theorem}
  If $b\not\in\bZ$, the operator $\Delta_\alpha^b$  {with domain $\xCinfc(M)$  is essentially self-adjoint in $L^2(M, d\omega)$} if $|\alpha|\ge 1$, and Proposition~\ref{prop:bp13-sa} still applies for $|\alpha|<1$.
  
  On the other hand, if $b\in\bZ$, Proposition~\ref{prop:bp13-sa} holds with the following change:
  if $-3<\alpha\le-1$ for every $k\neq b$ the operator $\hat\Delta_{\alpha,k}^b$ is essentially self-adjoint, while $\hat\Delta_{\alpha,b}^b$ is not.
\end{theorem}

The Aharonov-Bohm effect on the spectrum extends to this more general setting as follows.

\begin{theorem}
  For $\alpha>0$, the operator $-\Delta^b_\alpha$ on $L^2(M,d\omega_\alpha)$ has a non-empty discrete spectral component $\sigma_{\rm d}(-\Delta_\alpha^b)\subset[0,+\infty)$.
  When $b\in\bZ$ the operator has absolutely continuous spectrum $[0,+\infty)$ with embedded discrete spectrum.
  When $b\not\in \bZ$ the spectrum has no absolutely continuous part.
\end{theorem}

\begin{proof}
  For $b\neq k$, the spectrum of the operators  $\hat\Delta_{\alpha,k}^b$ (or of any of their self-adjoint extensions) is purely discrete (see e.g. \cite[Chapter 5]{T62}).  
  For $b = k$, on the other hand, the essential spectrum of $\hat\Delta_{\alpha,k}^b$ is non-empty and in particular it contains the half line $[0,+\infty)$ (see e.g. \cite[Theorem 15.3]{W87}).
\end{proof}

The previous theorems suggest that, for $b=0$ and $\alpha>0$, the $0$-th Fourier component, is the only responsible for the continuous spectrum.
The Aharonov-Bohm perturbation, when $b\in\bZ$, shifts this role to the $b$-th Fourier component.
When $b\notin\bZ$, no Fourier component produces a continuous spectrum.
This is a well-known phenomena in the case of asymptotically hyperbolic manifolds with finite volume \cite{GM07}, but completely new in this setting.

Further study of the cases $\alpha<0$ is outside the scope of this paper.
Note that the case $\alpha=-1$ considered on $\bR_+\times\bS^1$ coincides with the usual Aharonov-Bohm Laplacian in polar coordinates.
Moreover, in the case $\alpha=-1/2$, $-\Delta_\alpha^b$ has  discrete spectrum accumulating at $0$ and absolutely continuous spectrum in $[0,+\infty)$.
When $b\not\in\bZ$ an additional family of eigenvalues accumulating at $0$ appears.



\section{Proofs of the results on the Grushin cylinder}
\label{sec:grushin2d}

Using the Fourier decomposition with respect to\ the variable $\theta$ introduced in Section~\ref{sec:conic}, one gets 
  $\xLtwo(M_+,\vol) = \bigoplus_{k\in\bZ} H_k$,
  where
  $H_k \simeq \xLtwo\left(\bR_+,\frac 1 {x} dx\right)$.
The Laplace-Beltrami operator decomposes as $\lapl = \bigoplus_{k\in\bZ}
\widehat\lapl_k$, where
\(
  \widehat \lapl_k = \partial_x^2  - \frac{1}{x} \partial_x  - k^2 x^{2}
\).
Since $\lapl$ is essentially self-adjoint, each $\widehat\lapl_k$ is a
self-adjoint operator on the closure with respect to\ the graph norm of
$\xCinfc(\bR_+)$.
The unitary transformation $ U:\xLtwo\left(\bR_+,\frac 1 x dx
\right) \to \xLtwo(\bR_+,dx)$ defined by $Uv(x)\coloneqq \sqrt{x} v(x)$, then transforms the operator $\widehat\lapl_k$ into
\begin{equation}\label{eq:sep-grushin2d-deq}
  L_k \coloneqq
  U \widehat\lapl_k U^{-1} =
  \partial_x^2 - \frac 3 4 \frac 1 {x^2} - k^2 x^2,
  \quad \dom(L_k) = U \dom(\widehat\lapl_k).
\end{equation}

\subsection{Spectral properties of the Laplace-Beltrami operator}
\label{sec:spec-gr2d-pf}

Since the spectrum is invariant under unitary transformations, it's well known (see e.g. \cite{RS1})
that we can reduce the study of the spectrum of $\lapl$ to that of the operators $L_k$.

\begin{proof}[Proof of Theorem~\ref{thm:spec-grushin2d}]
  The operator $-L_0$ is the Schr\"odinger operator on the real line
  with a Calogero potential of strength $3/4$.  It is well-known that
  this operator has continuous spectrum $[0,+\infty)$, see e.g.
  \cite[Sec. VIII.10]{RS1}.

  Let now $k\neq0$. We want to compute the solutions of the eigenvalue
  problem
  \begin{equation}
    \label{eq:eigenvalue}
    (L_k -\lambda) u = 0 \iff 
    (\widehat\lapl_k-\lambda) U^{-1} u=0.
  \end{equation}
  With the change of variables $|k|x^2\mapsto z$ and multiplying by $ 4 k^2 z $, the previous equation reduces to
  \[
  \partial_z^2 v(z) + \left( -\frac{1}{4} + \frac{\lambda}{4z|k|} \right) v(z) = 0.
  \]
  This is the well-known Whittaker equation, whose solutions are the Whittaker functions
  $M_{\frac{\lambda}{4|k|,}\frac{1}{2}}(z)$ and
  $W_{\frac{\lambda}{4|k|},\frac{1}{2}}(z)$.
  The solutions of the eigenvalue problem \eqref{eq:eigenvalue}
  are then
  \[
  u_1(x)=\frac{ 1}{\sqrt x}M_{\frac{\lambda}{4|k|},\frac{1}{2}}(|k|x^2),
  \qquad
  u_2(x)=\frac{1}{\sqrt x} W_{\frac{\lambda}{4|k|},\frac{1}{2}}(|k|x^2).
  \]

  Through the asymptotic expansions of $M_{\nu,\mu}$ and $W_{\nu,\mu}$ (see e.g. \cite{B53}) one observes that $u_1$ is never square-integrable near infinity.
  On the other hand, $u_2\in\xLtwo(\bR_+)$ if and only if there exists a non-negative integer $\ell$ such that $-\ell = \frac 1 2 -\nu+\mu = \frac 1 2 - \frac{\lambda}{4|k|} + \frac 1 2$.
  Namely, for any $k\in \bN$ there exists a sequence $\{\lambda_{n,k} = 4|k|n \}_{n\in\bN}$ of eigenvalues with (non-normalized) eigenfunction $x\mapsto \psi_{n,k}(x)=W_{n,\frac{1}{2}}(|k|x^2)/\sqrt x$.

  Let $k=0$.
  Then, the operator $L_0$ given by \eqref{eq:sep-grushin2d-deq} can be interpreted as a Laplace operator with a relatively infinitesimally-bounded perturbation. 
  It is a well known result \cite{RS1} that its spectrum is purely absolutely continuous and equal to $[0,\infty)$.

  The statement now follows from the definition of $U^{-1}$ and standard spectral considerations (see e.g. \cite{RS1}).
\end{proof}

\begin{remark}
  \label{rmk:geneigen}
Observe that \eqref{eq:sep-grushin2d-deq} can be explicitly solved, 
it's solutions being of the form
\begin{equation}\label{eq:geneigenf}
\begin{cases}
c_1 x^{3/2} + \frac{c_2}{\sqrt{x}} & \mbox{for }\lambda = 0,\\
c_1 \sqrt{x}\; J_1(\sqrt{\lambda}x) + c_2 \sqrt{x}\; Y_1(\sqrt{\lambda}x) & \mbox{for }\lambda > 0,
\end{cases}
\end{equation}
  where $J_1$ and $Y_1$ are the Bessel functions of order $1$. 
  In particular, for $\lambda\geq0$ one has the explicit form of the generalised eigenfunctions of the absolutely continuous spectrum of $L_0$.
\end{remark}


\begin{proof}[Proof of Corollary~\ref{cor:counting-grushin2d}]
Let $N(E)$ be defined as in \eqref{eq:weyl-counting}.
By Theorem~\ref{thm:spec-grushin2d}, the following holds:
\begin{equation}
\label{eq:counting}
\#\{\lambda\in\sigma_p(-\lapl) \mid \lambda\le E\}
= \# \{ (n,k)\in \bN\times\bZ\setminus\{0\}\mid 4n|k|\le E\}.
\end{equation}
For fixed $k\in\bZ\setminus\{0\}$, this implies that the
couples $(n,k)$ in the r.h.s. above are such that $n\le E/(4|k|)$.  
Moreover, it is clear that for any $|k|>E/4$, no couple $(n,k)$ is admissible.
These facts and \eqref{eq:counting} imply
\[
N(E) = \sum_{0<|k|\leq\frac E 4} \frac E {4|k|}
     = \frac E 2 \sum_{\ell = 1}^{ \lfloor E/ 4 \rfloor } \frac 1 \ell.
\]

  It follows from the well-known asymptotic formula (see e.g. \cite{CG95})
  \begin{equation}
    \label{eq:estimateHarmonic}
      \sum_{m=1}^n \frac{1}{m} = \log(n) + \gamma + \frac1{2n} + O\left(\frac{1}{n^2}\right),
  \end{equation}
  where $\gamma$ is the Euler-Mascheroni constant, that as $E\rightarrow+\infty$
  \begin{align*}
  N(E) = 
  \frac{E}2\left(\log\left(\frac{E}4 + \gamma + O\left(\frac1E\right)\right)\right) 
  = \frac{E}2\log(E) + \left(\gamma-2\log(2)\right)\frac{E}2 + O(1).
  \end{align*}
\end{proof}

\subsection{Aharonov-Bohm effect}
\label{sec:aharonov-bohm-effect}

The magnetic Laplace-Beltrami operator on $M_+$ with Aharonov-Bohm vector potential
$A_b := -ib\,d\theta$, $b\in\bR$, is defined in \eqref{eq:gr-ab}.
After the transformation $U$, introduced above, we obtain the following operator acting on 
$\xLtwo(M_+,dx\,d\theta)$:
\[
L_b = U \lapl^b U^{-1}
    = \pa_x^2 - \frac{3}{4}\frac1{x^2} + x^{2}\left(\pa_\theta -ib\right)^2.
\]

  Through a straightforward extension of the proof of Theorem~\ref{thm:spec-grushin2d}, we immediately get Theorem~\ref{thm:spec-grushin2d-ab}. 
  For $b\in\bZ$ it is evident that the role of $L_0$ in the proof of Theorem~\ref{thm:spec-grushin2d} is now taken by $L_b$.

\begin{proof}[Proof of Corollary~\ref{cor:counting-grushin2d-ab}]
Without loss of generality we restrict ourselves to $b\in(-1/2,1/2)$, therefore $\kappa=0$.
Clearly, if $b=0$ the statement reduces to the one of Corollary~\ref{cor:counting-grushin2d}.
Assume $b\neq 0$.

Replacing $k$ with $|k-b|$ in the proof of Corollary~\ref{cor:counting-grushin2d} we observe that for $k=0$ the additional term $E/4|b|$ appears in the count. 
Thus, we can rewrite the counting function as
\[
N(E) = \frac{E}4 \sum_{k=1}^{\lfloor E/4\rfloor} \frac{1}{k+b} + \frac{E}{4|b|} + \frac{E}4 \sum_{k=1}^{\lfloor E/4\rfloor} \frac{1}{k-b}.
\]

We now apply the following identity (see e.g. \cite{OLBC10})
\[
\sum_{k=1}^{n} \frac{1}{k+x} = \psi(n+x+1) - \psi(1+x),
\]
and the asymptotic estimate 
\(
\psi(x+1) = \log(x) + \gamma + \frac1{2x} + O\left(\frac{1}{x^2}\right)
\) as $x\to\infty$, 
where $\psi(x)$ is the digamma function and $\gamma$ is the Euler-Mascheroni constant.
By a straightforward computation we obtain 
\begin{align*}
N(E) &= \frac{E}{4} \left(\psi\left(\lfloor E/4\rfloor + b + 1\right) - \psi(1+b)\right)
     + \frac{E}{4|b|} 
     + \frac{E}{4} \left(\psi\left(\lfloor E/4\rfloor - b + 1\right) - \psi(1-b)\right) \\
&= \frac{E}2 \log(E) + \frac{E}{2}\left(\frac1{2|b|} + \gamma -2\log(2) -\frac{\psi(1-b)+\psi(1+b)}{2}\right) + O(1).
\end{align*}

The general result then follows by shifting the above computation with $b \mapsto |\kappa-b|$.
\end{proof}

  We can now precisely determine the degeneracy of the eigenvalues as functions of $b$.

\begin{proof}[Proof of Theorem~\ref{thm:degenerspec-grushin2d-ab}]
The proof consists of three cases.

\textit{Case 1, $b\in\bR\setminus\bQ$:}
This immediately implies that $|k-b|\in\bR\setminus\bQ$. 
It is then straightforward to show that there exist no $(n',k') \neq (n,k)$ such that $\lambda^b_{n',k'} = \lambda^b_{n,k}$. 

\textit{Case 2, $b\in\bQ$:}
Let us write $b=p/q$ with $p,q\in \bZ$ such that $(p,q)=1$. 
Fix $(n,k)$ and $(n',k') \neq (n,k)$ such that $\lambda^b_{n,k}=\lambda^b_{n',k'}$. Then,
\begin{equation}
  \label{eq:eigenuguali}
  4 n' |qk' - p| = q\lambda^b_{n,k}.
\end{equation}
Without loss of generality assume that $qk' > p$. 
Then, since $4 n' |qk' - p|$ cannot divide $q$ because $(q,p)=1$, we have that it must divide $\lambda_{n,k}$.

From $q\neq1$, $\{|qk' - p| \mid k'\in\bZ \} \subseteq (q\bZ-p) \subsetneq \bZ$, we obtain that the number of couples $(n',k')$ such that $4n'|k'-b| = \lambda^b_{n,k}$ is bounded above by $2d(\lambda^b_{n,k}/4)$, where $d(n)$ denotes the number of divisors of $n$. 
In fact, if $|k'-b| = d_1$ for some $d_1\in\bQ$ divides $\lambda^b_{n,k}/4$, then $n' = \lambda^b_{n,k}/(4d_1)$. 
Observe that, due to the presence of a non integer $b$ in the term $|k-b|$, not all the possible divisors can be considered.
However, if a $k'> b$ can be taken, then there exists a $k'' < b$ that will give an additional couple $(k'',n')$.

\textit{Case 3, $b\in\bZ$:}
In this case, equation \eqref{eq:eigenuguali} reduces to $4 n' |k' - b| = \lambda_{n,k}$.
Then, for any $(n,k)$ with $k\neq b$, a simple computation shows that
\(
\lambda^b_{k, n + b} = \lambda^b_{n,k+b} 
= \lambda^b_{n,-k+b} = \lambda^b_{k, -n + b}
\).
If $n|k|$ is even, the combination $n=k=\lambda^K_{n,k+K}/8$ is repeated twice. This proves formula \eqref{eq:dlambda}.
Finally, this degeneracy cannot be achieved for $b \in\bQ\setminus\bZ$. 
In fact, it would require $\bZ\ni k' = (qn+p)/q$ which is impossible for $(q,p)=1$.
\end{proof}

Corollary~\ref{cor:counting-grushin2d-ab} suggests that in the limit $b\to k$, the number of eigenvalues in a finite interval explodes. Corollary~\ref{thm:decompspec-grushin2d-ab} makes this statement more precise, namely
\begin{itemize}
\item for any fixed $k\in\bZ$ and for any $n\in\bN$, the spacing 
\(
|\lambda^b_{n,k} - \lambda^b_{n-1,k}| \to 0 \mbox{ as } b\to k;
\)
\item for any fixed interval $I=[x_1,x_2]\subset[0,\infty)$ and any $N\in\bN$, 
\(
\#\{n\in\bN \mid \lambda_{n,k}^b \in I \} \geq N \mbox{ as } b\to k.
\)
\end{itemize}

\begin{proof}[Proof of Corollary~\ref{thm:decompspec-grushin2d-ab} (Corollary of Theorem~\ref{thm:spec-grushin2d-ab})]
Observe that
\begin{equation}\label{eq:spacing}
|\lambda^b_{n,k} - \lambda^b_{n-1,k}| = 4|k-b|.
\end{equation}
Taking the limit for $k\rightarrow b$ in the above yields immediately the first statement.

To prove the second statement, assume without loss of generality\ $k\geq 0$ and define
\[
L(b) := \left\lceil
\frac{x_1}{4|k-b|}
\right\rceil,
\quad
R(b) := \left\lfloor
\frac{x_2}{4|k-b|}
\right\rfloor.
\]
Then $\lambda^b_{L(b),k} \geq x_1$ and $\lambda_{R(b),k}^b \leq x_2$. 
If now 
\(
 |k-b| \leq \frac{x_2-x_1}{4 (N+1)},
\) 
by \eqref{eq:spacing} we have
\[
\#\{\lambda^b_{i,k} \mid L(b) \leq i \leq R(b)\} \geq N.
\]
This completes the proof of the second statement and hence of the corollary.
\end{proof}

This limiting process affects also the eigenfunctions.
Theorem~\ref{thm:degenef-grushin2d-ab} describes how the spectrum of the $k$-th Fourier components accumulates in the limit $b\to k$ and gives rise to the absolutely continuous part of the spectrum.

\begin{proof}[Proof of Theorem~\ref{thm:degenef-grushin2d-ab}]
Recall that $\psi_{n,k}^b(x,\theta) = e^{ik\theta}W_{n,\frac{1}{2}}(|k-b|x^2)/x$.
Since without loss of generality we can assume $k=0$, to complete the proof it suffices to show that 
\begin{equation}
  \label{eq:whittakerBessel}
  W_{n_j,\frac{1}{2}}(|b_j|x^2) \to \frac{\sqrt{\lambda}\, x}2 J_1(\sqrt{\lambda}x).
\end{equation}

Let us recall the following classical results (see resp.\ \cite{MOS66} and \cite{B53}).
\begin{gather*}
W_{n,1/2}(z) = (-1)^{n-1} z e^{-\frac12 z} L_{n-1}^1(z), \\
\lim_{n\to\infty} n^{-\alpha} L_n^1(x/n) = x^{-\frac12 \alpha} J_\alpha(2\sqrt{x}).
\end{gather*}
Here $L_n^\alpha$ is the generalised Laguerre polynomial of degree $n$ with parameter $\alpha$ and the limit is in the sense of uniform convergence on compact sets.

Define $n_j := 2j$ and $b_j := {\lambda}/{4(n_j+1)}$
so that $\lambda^{b_j}_{n_j+1,0} = \lambda$ for all $j>0$.
Then
\begin{align*}
\lim_{j\to \infty} W_{n_j,\frac{1}{2}}(|b_j|x^2) & = 
\lim_{j\to\infty} \frac{\lambda n_j x^2}{4 (n_j+1)}  
   \exp\left(-\frac1{2 n_j} \frac{\lambda n_j x^2}{4 (n_j+1)} \right)\,
   n_j^{-1} L^1_{n_j}\left(\frac1{n_j} \frac{\lambda n_j x^2}{4 (n_j+1)} \right) \\
   &= \frac{\sqrt{\lambda}\, x}2 J_1(\sqrt{\lambda}x).
\end{align*}
This completes the proof of \eqref{eq:whittakerBessel}.
\end{proof}




\section{Proofs of the results on the Grushin sphere}
\label{sec:quantum}

As for the Grushin cylinder, we can separate the space using the
orthonormal eigenbasis of $\bS^1$ getting
\(
  L^2(S_+,d\omega)
  = \bigoplus_{k=-\infty}^{\infty} H^{S_+}_k
\), where \(
  H_k^{S_+}\simeq L^2([0,\pi/2),\tan(x)dx).
\) 
On each $H^{S_+}_k$ the operator separates as  
\[
  \tilde\Delta_k := \pa_x^2 - \frac1{\sin(x)\cos(x)} \pa_x - \tan^2(x) k^2.
\]

\subsection{Spectral properties of the Laplace-Beltrami operator}
\label{sec:spectralGrushinS}

We can mimic the steps presented in Section \ref{sec:spec-gr2d-pf} to analyse the spectrum of $\Delta$.

\begin{proof}[Proof of Theorem~\ref{thm:spec-grushinS}]
  We look for solutions $\phi\in H^{S_+}_k$ of the eigenvalue equation
  \(
  -\tilde\Delta_k \phi(x) = \lambda\phi(x)
  \).
  Since $k$ appears in $\tilde\Delta_k$ only squared, the eigenvalues are symmetric with respect to $k=0$.
  To simplify the notation, in the following we will assume $k\ge 0$. The same considerations hold for $k<0$ by replacing $k$ with $|k|$.

  With the change of variables $z=\cos(x)^2$ and writing $\phi(x) = (-z)^{\frac k2}\varphi(z)$,
  the eigenvalue equation becomes
  \[
  4(-z)^{\frac k2}\left(
    z(1-z) \pa_z^2 \varphi(z) + (1+k)(1-z) \pa_z \varphi(z) + \frac\lambda4 \varphi(z)
  \right) = 0.
  \]
  The equation in bracket is a particular example of the well-known Euler's hypergeometric equations. 
  Two linearly independent solutions can be found in terms of Gauss Hypergeometric Functions  $F(a,b;c;z)$ (see \cite[Vol. 1, Ch. 2]{B53}) as follows:
  \begin{align*}
  \phi_1(x) &= i^{-k} \cos(x)^{-k} F\left( 
      -\frac{k}2 - \frac{\sqrt{\lambda+k^2}}2, 
      -\frac{k}2 + \frac{\sqrt{\lambda+k^2}}2;
      1 - k;
      \cos(x)^2
    \right), \\
  \phi_2(x) &= i^{k} \cos(x)^{k} F\left( 
      \frac{k}2 - \frac{\sqrt{\lambda+k^2}}2, 
      \frac{k}2 + \frac{\sqrt{\lambda+k^2}}2;
      1 + k;
      \cos(x)^2
    \right).
  \end{align*}
  Notice here that in the case 
  $\frac{k}2 \pm\frac{\sqrt{\lambda+k^2}}2, k-1 \in \bN_0$ the first solution is not defined, in fact we are in the so called \emph{degenerate case} and the only regular solution is $\phi_2$. 
  Therefore we do not need to introduce the other corresponding linearly independent solution.

  A solution is an eigenfunction of the Laplace-Beltrami operator if it is in
  $H_k^{S_+}$. 
  For this to be true, the solutions has to be square-integrable near $0$ with respect to the measure $d\omega := \tan(x)^{-1} dx$. 
  Equivalently, the solutions have to be $O(x)$ as $x \to 0$ and, in particular, they must vanish at zero.

  Let us recall that
  \begin{align}
    \label{eq:asympPhi1}
    \phi_1\left( 0 \right) &= i^{-k} \frac{\Gamma (1-k)}{\Gamma \left(-\frac{k}{2}-\frac{\sqrt{k^2+\lambda }}{2}+1\right) \Gamma \left(-\frac{k}{2}+\frac{\sqrt{k^2+\lambda }}{2}+1\right)}, \\
    \label{eq:asympPhi2}
    \phi_2\left( 0 \right) &= i^{k} \frac{\Gamma (k+1)}{\Gamma \left(\frac{k}{2}-\frac{\sqrt{k^2+\lambda }}{2}+1\right) \Gamma \left(\frac{k}{2}+\frac{\sqrt{k^2+\lambda }}{2}+1\right)}.
  \end{align}

  To see that $\phi_1(0)\neq 0$ if $k \geq 1$, it is enough to notice that
  \[
   \pm \frac{k}2 + \frac{\sqrt{k^2+\lambda}}2 \geq 0
   \mbox{ for all $k\in\bN_0$ and $\lambda\in\bR_+^0$.}
  \]
  By the previous considerations, $\phi_1\notin H_k^{S_+}$ for $k\ge1$.  
  Since $k=0$ corresponds to the degenerate case, where the two solutions coincide, in the following we consider only $\phi_2$.
  
  By \eqref{eq:asympPhi2}, in order for $\phi_2(0)=0$ to hold there has to exist $n\in \bN_0$ such that $\lambda$ satisfies
  \[
    \frac{k}2 + 1 - \frac{\sqrt{k^2+\lambda}}2 = -n.
  \]
  Solving the above for $\lambda$, yields the following expression for the candidate eigenvalue
  \[
  \lambda = \lambda^+_{k,n} := 4(1+n)(1+n+k).
  \]
  
  In order to prove that the candidate eigenvalues $\lambda^+_{k,n}$ are indeed eigenvalues, we check the order of convergence of the solutions. 
  For this purpose we use the well-known identity \cite[15.2(ii)]{OLBC10} for $a=-m\in\bZ_-\!\cup\{0\}$, $b > 0$ and $c\not\in\bZ_-\!\cup\{0\}$ 
  \begin{align}
    \label{eq:expPhi}
    F(-m, b; c; z) = \sum_{\ell=0}^m (-1)^m \binom{m}{\ell} \frac{(b)_\ell}{(c)_\ell} x^\ell.
  \end{align}
  Replace the values of the parameters for $\phi_2$ in the above, and set $\lambda = \lambda^+_{k,n}$, to obtain
  \[
  F\left(-(n+1), n+k+1; k+1; \cos(x)^2\right) = \sum_{\ell=0}^{n+1} (-1)^{n+1} \binom{n+1}{\ell} \frac{(n+k+1)_\ell}{(k+1)_\ell} \cos(x)^{2\ell}.
  \]
  I.e. $\phi_2$ and his derivative have the correct behaviour at $0$ and are regular at $\pi/2$, proving that $\lambda^+_{k,n}$ are indeed eigenvalues.

  In order to obtain the expression of the eigenvalues and eigenfunctions given in the statement, it suffices to replace $n+1$ with $n$ in the definition of $\lambda^b_{n,k}$.
  The theorem then follows by the symmetry of the problem with respect to\ $k=0$.
\end{proof}

We are now in a position to derive the Weyl's law for the Laplace-Beltrami operator of the Grushin sphere.

\begin{proof}[Proof of Corollary~\ref{cor:counting-grushinS}]
  By the symmetry of the eigenvalue problem with respect to\ $k=0$, it follows that $N(E)=2\cdot\#\{k\in\bN,\, n\in\bN \mid \lambda_{n,k} \le E \} + \#\{n\in\bN\mid\lambda_{n,0}\le E\}$.  
  Let $N_0(E)$ be the counting function for this second sum. 
  It is easy to see that $\lambda(0,n) \leq E$ for $n\in\left[0,\lfloor \sqrt{E}/2 \rfloor\right]$. Therefore $N_0(E) = O(\sqrt{E})$.
  
  Let $N_+(E)$ be the counting function for positive values of $k$ and define
  \begin{equation}\label{eq:K(n)}
  K(n) := \frac{E-4n^2}{4n}.
  \end{equation}
  A simple computation shows that $\lambda(k,n)\le E$ if and only if
  \(
  0 < k \leq \left\lfloor K(n) \right\rfloor.
  \) 
  Additionally, notice that if $n>\lfloor\sqrt{E}/2 \rfloor =: \eta_1(E)$, then $K(n) < 0$.
  
  The simple bound
  \(
  \lfloor K(n) \rfloor \leq \#\left\{ k\in\left[
      0, K(n) 
    \right]\cap\bN \right\} \leq \lceil K(n) \rceil,
  \)
  now implies that 
  \[
  \sum_{n=1}^{\eta_1(E)} \lfloor K(n) \rfloor \leq N_+(E) \leq \sum_{n=1}^{\eta_1(E)} \lceil K(n) \rceil.
  \]

  Due to the asymptotic estimate \eqref{eq:estimateHarmonic}, we immediately get that, as $E\rightarrow+\infty$,
  \[
  \begin{split}
    N(E) ={}& 2 N_+(E) + N_0(E) = 2\sum_{n=1}^{\eta_1(E)} K(n) + O(\sqrt{E}) \\
        ={}& \frac{E}2 \sum_{n=1}^{\eta_1(E)} \frac{1}{n} - 2 \sum_{n=1}^{\eta_1(E)} n + O(\sqrt E)\\
    ={}& \frac{E}{2}\left( \log(\sqrt{E}/2) + \gamma \right) -\frac{E}{4} + O(\sqrt E)\\
    ={}& \frac{E}4\log(E) + \left( \gamma - \log(2) - \frac{1}2 \right) \frac{E}2 + O(\sqrt E).
  \end{split}
  \]  
\end{proof}

It follows from Theorem~\ref{thm:spec-grushinS} that for $k\neq 0$ the operator $\tilde\Delta_k$ acting on $H_k$ presents an infinite amount of eigenvalues accumulating at infinity that can be explicitly described by
\[
  \sigma_d(H_k) := \{ \lambda_{n,|k|} = 4(1+n)(1+n+|k|) \mid n\in\bN, k\in\bZ\}.
\]

Due to the symmetry with respect to\ $k$ of $\lambda_{n,|k|}$, all the eigenvalues are at least double degenerate.
Moreover, this degeneracy must be finite.
Indeed, it is enough to observe that  each operator has a ground state of energy greater than 
\(
\lambda_{1,|k|} = 4(|k|+1)
\),
and that the function $n\mapsto\lambda_{n,|k|}$ is increasing in $n\in\bN$. 

The degeneracy of an eigenvalue $\lambda$ can be easily bounded above by $(\lambda-4)/2$, but this is far from being optimal.
In fact, the computation of the first five million eigenvalues (see Figure~\ref{fig:evDegS}) suggests the growth of the degeneracy to be irregular and slow as in the case of the Grushin cylinder (see Theorem~\ref{thm:degenerspec-grushin2d-ab}).

\begin{figure}
\centering
\includegraphics[width=0.7\linewidth]{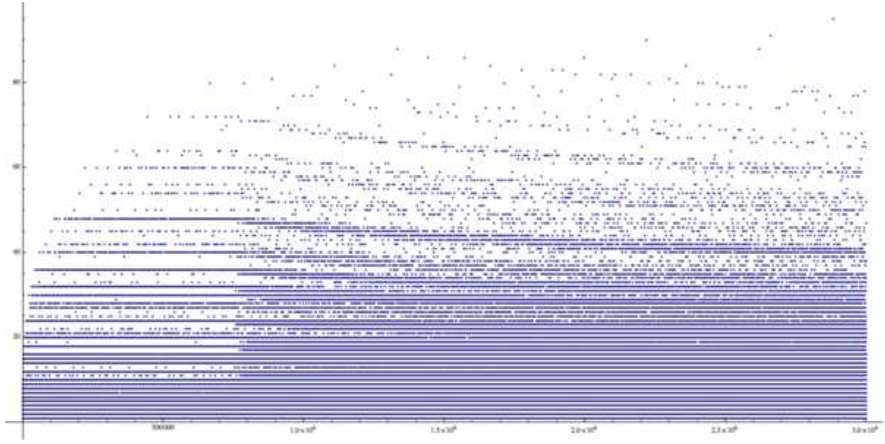}
\caption{\small Degeneracy (halved) of the first $4.893.535$ eigenvalues (namely $\lambda_{n,k} < 3\cdot 10^6$). Observe that only one of those eigenvalues attains the higher multiplicity of $110$, far below than our upper bound.}
\label{fig:evDegS}
\end{figure}

Unfortunately, it is not possible to obtain a more precise description of the degeneracy with the simple techniques employed in Theorem~\ref{thm:degenerspec-grushin2d-ab}. 
Indeed, in this case the problem reduces to counting the number of solutions of a non-linear Diophantine equation, which is well-known to be an hard problem.

\subsection{Aharonov-Bohm effect}
\label{sec:ab-grushinSphere}

The magnetic Laplace-Beltrami operator induced by the magnetic vector potential $A_b = -ib\; d\phi$, $b\in\bR$, on $S_+^\circ$ is given in \eqref{eq:ABlaplace-sphere}.
Its domain is the closure of $\xCinfc(S^\circ_+)$ with respect to\ the Sobolev norm $W^{1,2}_0(S^\circ_+)$.
Its Fourier components are
\[
\tilde\Delta^b_k = \pa_x^2 - \frac{1}{\sin(x)\cos(x)}\pa_x - \tan(x)^2(k-b)^2.
\]

Theorem~\ref{thm:spec-grushinS-ab} follows immediately from the arguments of Theorem~\ref{thm:spec-grushinS} with $k$ replaced by $k-b$.

\begin{proof}[Proof of Corollary~\ref{cor:counting-grushinS-ab}]
  Without loss of generality, we can assume $b\in (-1/2,1/2)$, i.e. $\kappa=0$.
  If $b=0$, the statement reduces to the one of Corollary~\ref{cor:counting-grushinS}. Assume, by the symmetry of the eigenvalue expression, that $b<0$.
  
  As in the proof of Corollary~\ref{cor:counting-grushinS}, we split the counting function in two components $N_-(E)$ and $N_+(E)$, depending on wether $k$ is smaller or bigger than $b$.
  The estimates on $N_+(E)$ and $N_-(E)$ are then obtained following the proof of  Corollary~\ref{cor:counting-grushinS} with
  \[
  K(n) = \frac{E-4n^2}{4n} +|b|.
  \]
  Since the sum has to be computed for \(n\le \eta_1(E) :=
   \left\lfloor \frac{|b|+\sqrt{E+|b|^2}}{2} \right\rfloor,
  \) 
  it is easy to see that $b$ only appears (linearly) in the $O(\sqrt E)$ term.
  Since $|b|\le 1/2$ this completes the proof.
\end{proof}

As already anticipated in the previous section, the degeneracy of the spectrum 
for the Grushin sphere seems to be of similar nature as for the Grushin cylinder,
at least from a numerical point of view, but having a precise control on it is much more involved and probably not possible at present.
Nevertheless, one can still prove that the degeneracy is very unstable with
respect to the parameter $b$ and, in particular, that the spectrum is simple for 
$b\in\bR\setminus\bQ$ and finitely degenerate for $b\in\bQ$.
This is summarised in Corollary~\ref{cor:counting-grushinS-ab} and it follows 
from an argument very close to the one in the proof of 
Corollary~\ref{thm:degenerspec-grushin2d-ab}.






\appendix

\addcontentsline{toc}{section}{Aknowledgements}
\section*{Aknowledgements}
The research of M.\ Seri has been supported by the EPSRC grant EP/J016829/1.
The research of U.\ Boscain and D.\ Prandi has been supported by the European Research Council, ERC StG 2009 ``GeCoMethods'', contract number 239748.
U.\ Boscain has also been supported by EU FP7 Project QUAINT, grant agreement np. 297861.
D.\ Prandi has also been supported by the Laboratoire d'Excellence Archim\`ede, Aix-Marseille Universit\'e.

The authors would like to thank professor G.\ Panati for the fruitful discussions.


\begin{thebibliography}{10}

\bibitem{AT98}
R.~Adami and A.~Teta.
\newblock On the {A}haronov-{B}ohm {H}amiltonian.
\newblock {\em Lett. Math. Phys.}, 43(1), 1998.

\bibitem{AgrBarBoscbook}
A.~Agrachev, D.~Barilari, and U.~Boscain.
\newblock {Introduction to Riemannian and sub-Riemannian geometry (Lecture
  Notes)}.
\newblock \url{http://people.sissa.it/agrachev/agrachev_files/notes.html},
  2012.

\bibitem{gaussbonnet1}
A.~Agrachev, U.~Boscain, and M.~Sigalotti.
\newblock A {G}auss-{B}onnet-like formula on two-dimensional
  almost-{R}iemannian manifolds.
\newblock {\em Discrete Contin. Dyn. Syst.}, 20(4), 2008.

\bibitem{gaussbonnet2}
A.~A. Agrachev, U.~Boscain, G.~Charlot, R.~Ghezzi, and M.~Sigalotti.
\newblock Two-dimensional almost-{R}iemannian structures with tangency points.
\newblock {\em Ann. Inst. H. Poincar\'e Anal. Non Lin\'eaire}, 27(3), 2010.

\bibitem{A76}
T.~M. Apostol.
\newblock {\em Introduction to analytic number theory}.
\newblock Springer-Verlag, New York-Heidelberg, 1976.
\newblock Undergraduate Texts in Mathematics.

\bibitem{baouendi}
M.~S. Baouendi.
\newblock Sur une classe d'op\'erateurs elliptiques d\'eg\'en\'er\'es.
\newblock {\em Bull. Soc. Math. France}, 95, 1967.

\bibitem{B53}
H.~Bateman and A.~Erdelyi.
\newblock {\em Higher Transcendental Functions. Vol. 1--3}.
\newblock McGraw-Hill, 1953--1956.

\bibitem{bellaiche}
A.~Bella{\"{\i}}che.
\newblock The tangent space in sub-{R}iemannian geometry.
\newblock In {\em Sub-{R}iemannian geometry}, volume 144 of {\em Progr. Math.}
  Birkh\"auser, Basel, 1996.

\bibitem{tannaka}
B.~Bonnard, J.-B. Caillau, R.~Sinclair, and M.~Tanaka.
\newblock Conjugate and cut loci of a two-sphere of revolution with application
  to optimal control.
\newblock {\em Ann. Inst. H. Poincar\'e Anal. Non Lin\'eaire}, 26(4), 2009.

\bibitem{B07}
D.~Borthwick.
\newblock {\em Spectral theory of infinite-area hyperbolic surfaces}, volume
  256 of {\em Progress in Mathematics}.
\newblock Birkh\"auser Boston, Inc., Boston, MA, 2007.

\bibitem{q1}
U.~Boscain, T.~Chambrion, and G.~Charlot.
\newblock Nonisotropic 3-level quantum systems: complete solutions for minimum
  time and minimum energy.
\newblock {\em Discrete Contin. Dyn. Syst. Ser. B}, 5(4), 2005.

\bibitem{BCha}
U.~Boscain, G.~Charlot, J.-P. Gauthier, S.~Gu{\'e}rin, and H.-R. Jauslin.
\newblock Optimal control in laser-induced population transfer for two- and
  three-level quantum systems.
\newblock {\em J. Math. Phys.}, 43(5), 2002.

\bibitem{arsnormal}
U.~Boscain, G.~Charlot, and R.~Ghezzi.
\newblock Normal forms and invariants for 2-dimensional almost-{R}iemannian
  structures.
\newblock {\em Differential Geom. Appl.}, 31(1), 2013.

\bibitem{BL11}
U.~Boscain and C.~Laurent.
\newblock The {L}aplace-{B}eltrami operator in almost-{R}iemannian geometry.
\newblock {\em Annales de l'institut Fourier}, 63(5), 2013.

\bibitem{BP13}
U.~Boscain and D.~Prandi.
\newblock The heat and schr\"odinger equations on conic and anticonic-type
  surfaces.
\newblock {\em Preprint arXiv:1305.5271v2}, 2013.

\bibitem{cdv15}
Y.~{Colin de Verdi{\`e}re}, L.~{Hillairet}, and E.~{Tr{\'e}lat}.
\newblock {Quantum ergodicity and quantum limits for sub-Riemannian
  Laplacians}.
\newblock {\em ArXiv e-prints}, 2015.

\bibitem{CG95}
J.~H. Conway and R.~K. Guy.
\newblock {\em The book of numbers}.
\newblock Copernicus, New York, 1996.

\bibitem{dOP08}
C.~R. de~Oliveira and M.~Pereira.
\newblock Mathematical justification of the {A}haronov-{B}ohm {H}amiltonian.
\newblock {\em J. Stat. Phys.}, 133(6), 2008.

\bibitem{FL1}
B.~Franchi and E.~Lanconelli.
\newblock Une m\'etrique associ\'ee \`a une classe d'op\'erateurs elliptiques
  d\'eg\'n\'er\'es.
\newblock In {\em Conference on linear partial and pseudodifferential operators
  (Torino, 1982)}. Rend. Sem. Mat. Univ. Politec. Torino, (Special Issue),
  1984.

\bibitem{ghezzi15}
R.~{Ghezzi} and F.~{Jean}.
\newblock {Hausdorff volume in non equiregular sub-Riemannian manifolds}.
\newblock {\em ArXiv e-prints}, Jan. 2015.
\newblock {To appear on Nonlinear Analysis: Theory, Methods \& Applications}.

\bibitem{GM07}
S.~Gol{\'e}nia and S.~Moroianu.
\newblock Spectral analysis of magnetic {L}aplacians on conformally cusp
  manifolds.
\newblock {\em Ann. Henri Poincar\'e}, 9(1), 2008.

\bibitem{grushin1}
V.~V. Gru{\v{s}}in.
\newblock A certain class of hypoelliptic operators.
\newblock {\em Mat. Sb. (N.S.)}, 83 (125), 1970.

\bibitem{hmmrev}
A.~Hassell, R.~Mazzeo, and R.~B. Melrose.
\newblock Analytic surgery and the accumulation of eigenvalues.
\newblock {\em Comm. Anal. Geom.}, 3(1-2), 1995.

\bibitem{I80}
V.~J. Ivri{\u\i}.
\newblock The second term of the spectral asymptotics for a
  {L}aplace-{B}eltrami operator on manifolds with boundary.
\newblock {\em Funktsional. Anal. i Prilozhen.}, 14(2), 1980.

\bibitem{jrev}
C.~M. Judge.
\newblock Tracking eigenvalues to the frontier of moduli space. {I}.
  {C}onvergence and spectral accumulation.
\newblock {\em J. Funct. Anal.}, 184(2), 2001.

\bibitem{MOS66}
W.~Magnus, F.~Oberhettinger, and R.~P. Soni.
\newblock {\em Formulas and theorems for the special functions of mathematical
  physics}.
\newblock Third enlarged edition. Die Grundlehren der mathematischen
  Wissenschaften, Band 52. Springer-Verlag New York, Inc., New York, 1966.

\bibitem{montrev}
R.~Montgomery.
\newblock Hearing the zero locus of a magnetic field.
\newblock {\em Comm. Math. Phys.}, 168(3), 1995.

\bibitem{montgomery}
R.~Montgomery.
\newblock {\em A tour of subriemannian geometries, their geodesics and
  applications}, volume~91 of {\em Mathematical Surveys and Monographs}.
\newblock American Mathematical Society, Providence, RI, 2002.

\bibitem{OLBC10}
F.~W.~J. Olver, D.~W. Lozier, R.~F. Boisvert, and C.~W. Clark, editors.
\newblock {\em N{IST} handbook of mathematical functions}.
\newblock U.S. Department of Commerce, National Institute of Standards and
  Technology, Washington, DC; Cambridge University Press, Cambridge, 2010.
\newblock With 1 CD-ROM (Windows, Macintosh and UNIX).

\bibitem{P95}
L.~B. Parnovski.
\newblock Asymptotics of {D}irichlet spectrum on some class of noncompact
  domains.
\newblock {\em Math. Nachr.}, 174, 1995.

\bibitem{RS1}
M.~Reed and B.~Simon.
\newblock {\em Methods of modern mathematical physics. {I}}.
\newblock Academic Press, Inc. [Harcourt Brace Jovanovich, Publishers], New
  York, second edition, 1980.
\newblock Functional analysis.

\bibitem{S02}
E.~Shargorodsky.
\newblock Semi-elliptic operators generated by vector fields.
\newblock {\em Dissertationes Math. (Rozprawy Mat.)}, 409, 2002.

\bibitem{T62}
E.~Titchmarsch.
\newblock {\em Eigenfunction expansions. Part 1. Second Edition}.
\newblock Oxford University Press, 1962.

\bibitem{W87}
J.~Weidmann.
\newblock {\em Spectral theory of ordinary differential operators}, volume 1258
  of {\em Lecture Notes in Mathematics}.
\newblock Springer-Verlag, Berlin, 1987.

\end{thebibliography}

\def\cprime{$'$}

\end{document}